\documentclass[12pt]{amsart}
\usepackage{amsmath,amsthm,amsfonts,amssymb}
\usepackage{amsmath, amsthm, amssymb, amscd, amsfonts}
\usepackage[all]{xy}
\usepackage{amssymb, amsthm, amsmath}
\usepackage[english]{babel}
\usepackage{amsfonts}
\usepackage{color}
\usepackage{leftidx}
\usepackage{enumerate}

\usepackage[dvipsnames]{xcolor}

\theoremstyle{plain}

\newtheorem{theorem}{Theorem}[section]
\newtheorem{corollary}[theorem]{Corollary}
\newtheorem{conjecture}[theorem]{Conjecture}
\newtheorem{proposition}[theorem]{Proposition}
\newtheorem{lemma}[theorem]{Lemma}
\newtheorem{question}[theorem]{Question}

\theoremstyle{definition}

\newtheorem{example}[theorem]{Example}

\newtheorem{rem}[theorem]{Remark}
\theoremstyle{remark}
\newtheorem{remark}[theorem]{Remark}

\DeclareMathOperator{\id}{id}
\DeclareMathOperator{\Hom}{Hom}

\newcommand{\ra}{\rightarrow}

\newcommand{\ot}{\otimes}
\newcommand{\mtc}{\mathcal}

\newcommand{\lb}{\label}
\newcommand{\Lam}{\Lambda}

\newcommand{\eps}{\fp}
\newcommand{\bn}{\begin}

\newcommand{\C}{\mtc{C}}\newcommand{\E}{\mtc{E}}\newcommand{\B}{\mtc{B}}
\newcommand{\Rp}{\mtr{Rep}}
\newcommand{\A}{\mtc{A}}
\newcommand{\SC}{\Lambda_{\C}}
\newcommand{\SD}{\Lambda_{\mtc{D}}}

\newcommand{\D}{\mtc{D}}
\newcommand{\M}{\mtc{M}}

\numberwithin{equation}{section}

\newcommand{\Rg}{{\mathrm R}}

\newcommand{\fp}{\mtr{FPdim}}
\newcommand{\bp}{\bn{proposition}}
\newcommand{\ep}{\end{proposition}}
\newcommand{\gr}{K_0}
\newcommand{\rD}{\Rg_{ _{\D}}}
\newcommand{\rC}{\Rg_{ _{\C}}}
\newcommand{\rE}{\Rg_{ _{\E}}}

\newcommand{\mtr}{\mathrm}

\newcommand{\cc}{\mathcal{C}}

\newcommand{\dd}{\mathcal{D}}

\newcommand{\zz}{\mathcal{Z}}

\newcommand{\comod}{\mathrm{comod}}

\newcommand{\iso}{\stackrel{\sim}{\longrightarrow}}

\newcommand{\bra}[1]{\langle #1 \rangle}

\newcommand{\kk}{\Bbbk}

\newcommand{\ti}{\mbox{-}\,}

\newcommand{\un}{\mathbf{1}}

\newcommand{\Aut}{\mathrm{Aut}}
\newcommand{\ncm}{\newcommand}
\ncm{\bpf}{\bn{proof}}
\ncm{\epf}{\end{proof}}
\numberwithin{equation}{section}
\ncm{\md}{\medbreak}
\ncm{\cx}{\mtc{X}}
\ncm{\core}{\mtr{core}}
\ncm{\cs}{\mtc{S}}

                  \ncm{\cd}{\mtc{D}}
\ncm{\onh}{On the other hand\;}
\ncm{\cm}{\mtc{M}}
\ncm{\cz}{\mtc{Z}}
\ncm{\inv}{^{-1}}
\ncm{\ene}{\end{enumerate}}
\ncm{\os}{\oplus}
\ncm{\Irr}{\mathrm{Irr}}
\ncm{\hsa}{Hopf subalgebra of }
\ncm{\ses}{semisimple}
\ncm{\x}{$}
\ncm{\mi}{\mtr{I}}
\ncm{\cZ}{\mtc{Z}}\ncm{\xra}{\xrightarrow}
\ncm{\cb}{\mtc{B}}\ncm{\ca}{\mtc{A}}
\ncm{\irr}{\Irr}
\ncm{\co}{\mtc{O}}
\ncm{\ce}{\mtc{E}}
\ncm{\cg}{{\mtr{K}_0}}\ncm{\ci}{\mtc{I}}
\ncm{\blue}{\textcolor[rgb]{.00, .00, 1.00}}
\ncm{\red}{\textcolor[rgb]{1.00, .00, .00}}
\ncm{\green}{\textcolor[rgb]{.00, 1.00, .00}}
\ncm{\Gm}{\Gamma}\ncm{\ind}{\mtr{Ind}}\ncm{\res}{\mtr{Res}}
\ncm{\ed}{\end{document}}
\ncm{\beq}{\begin{equation}}
\ncm{\beqn}{\begin{equation*}}
\ncm{\eeq}{\end{equation}}
\ncm{\eeqn}{\end{equation*}}
\ncm{\bea}{\begin{eqnarray}}
\ncm{\eea}{\end{eqnarray}}
\ncm{\beanon}{\begin{eqnarray*}}
\ncm{\eeanon}{\end{eqnarray*}}\ncm{\ek}{\eps|_K}\ncm{\diez}{\#}\ncm{\bl}{\bn{lemma}}
\ncm{\el}{\end{lemma}}\ncm{\Inv}{\mtr{Inv}}\ncm{\rad}{\mtr{rad}}\ncm{\br}{\bn{rem}}
\ncm{\er}{\end{rem}}\ncm{\rep}{\Rp}
\ncm{\bc}{\bn{corollary}}
\ncm{\ec}{\end{corollary}}
\numberwithin{equation}{section}
\ncm{\bt}{\begin{theorem}}
\ncm{\et}{\end{theorem}}
\ncm{\bqst}{\begin{question}}
\ncm{\eqst}{\end{question}}
\ncm{\bconj}{\begin{conjecture}}
\ncm{\econj}{\end{conjecture}}\ncm{\ct}{\cite}

\newcommand{\Zed}{\mathbb{Z}}

\newcommand{\Comp}{\mathbb{C}}
\newcommand{\NN}{\mathbb{N}}
\newcommand{\IA}{\mathbb{A}}

\newcommand{\Spec}{\mathrm{Spec}}
\newcommand{\End}{\mathrm{End}}
\newcommand{\ldual}[1]{\leftidx{^*}{\!#1}{}}
\newcommand{\rdual}[1]{{#1}^*}

\newcommand{\KER}{\mathfrak{Ker}}
\newcommand{\vect}{\mtr{Vec}}

\title[Fusion categories]
{On normal tensor functors and coset decompositions for fusion categories}
\author{A. Brugui\`eres and Sebastian  Burciu}%\thanks{The work of S.B. was partially supported by the Romanian National Authority for Scientific Research, CNCS-UEFSCDI, grant no. 88/05.10.2011}%project number PN II - ID - PCE-2011-3-0039}
\address{D\'epartement de Math\'ematiques\\Universit\'e Montpellier II \\
34095 Montpellier Cedex 5, France}\email{bruguier@math.univ-montp2.fr, alain.bruguieres@gmail.com}\address{\textnormal{and}} \address{Simion Stoilow Institute of Mathematics of the Romanian Academy, Research Unit 5, P.O. Box 1-764, RO-014700, Bucharest, Romania}%\address{\textnormal{and}} \address{University of Bucharest, Faculty of Mathematics and Computer Science, 14 Academiei St., Bucharest, Romania }
\email{sebastian.burciu@imar.ro, smburciu@gmail.com}
\date{\today}
\begin{document}
\maketitle
\begin{abstract} We introduce the notion of double cosets relative to two fusion subcategories of a fusion category. Given a tensor functor $F : \C \to \D$ between fusion categories, we introduce an equivalence relation $\approx^F$ on the set $\Lambda_\C$ of isomorphism classes of simple objects of $\C$, and when $F$ is dominant, an equivalence relation $\approx_F$ on $\Lambda_\D$. We show that the equivalent classes of $\approx^F$ are cosets.
We also give a description of the image of $F$ when it is a normal tensor functor, and we show that $F$ is normal if and only if the images of $\approx^F$ equivalent elements of $\Lambda_\C$ are colinear. We study the situation where the composition of two tensor functors $F=F'F''$ is normal, and we give a criterion of normality for $F''$, with an application to equivariantizations.
Lastly, we introduce the  radical of a fusion subcategory and compare it to its commutator in the case of a normal subcategory.
%The construction of the double cosets is achieved by applying Frobenius-Perron theory for some nonnegative matrices associated to certain special operators on the Grothendieck ring of the fusion category.
We also give a description for the image of a normal tensor functor between any two fusion categories.
\end{abstract}
%\tableofcontents
\section*{Introduction}
In this paper we introduce  the notion of double cosets in a fusion category relative to two fusion subcategories,  which generalize double cosets in a group relative to two subgroups,  and we use this notion to study tensor functors between fusion categories.
Double cosets can be studied in terms of the corresponding regular elements in the Grothendieck ring of the fusion category.

\md
 We show that a tensor functor $F:\C\ra \D$ between fusion categories gives rise to two equivalence relations: one, denoted by $\approx^F$, on the set of simple objects $\Lambda_\C$ of isomorphism classes of simple objets of $\C$, and another (if $F$ is dominant), denoted by $\approx_F$, on the set  $\Lambda_\D$  of simple objects of $\D$.  These equivalence relations are the categorical analogues of  those  introduced by Rieffel in \cite{Ri} for the restriction functor attached to an extension of semisimple rings. We prove that the equivalence classes for $\approx^F$ are left (and right) cosets relative to a certain fusion subcategory of $\C$.
\md
In the special case of a normal tensor functor (a notion introduced in \cite{brn}); we give
a description of images of objects under normal tensor functors $F:\C\ra \D$. In particular,  we show that normal functors are characterized by the fact that objects in the same equivalence class in $\Lambda_\cc$ have colinear images in the Grothendieck ring of $\D$. A similar result holds for the equivalence classes in $\Lambda_\D$ when $F$ is dominant.
\md
We study the situation when a composite of tensor functors  $F = F'F''$ is normal; we show that if $F''$ is dominant and and $F$ normal then $F'$ is also normal. We also give a criterion of normality for $F''$ in Theorem \ref{thm-normal}. As an illustration, we apply this result to equivariantizations. Denote by $\C^G$ the equivariantization of a fusion category $\C$ under the action of a finite group $G$ by tensor autoequivalences; it is again a fusion category under some reasonable hypotheses ({\it e. g.} over~$\Comp)$. We obtain that if $H$ is a subgroup of $G$, then the restriction functor $\C^G \to \C^H$ is normal if and only if $H$ is a normal subgroup of $G$, and in that case, $\C^G$ is equivalent to $(\C^H)^{G/H}$ for a certain action of $G/H$ on $\C^H$ by tensor automorphisms.
\md
Lastly, by analogy with ring theory, we define the radical of a fusion subcategory of a fusion category $\C$. We show that the radical of a normal subcategory coincides with its commutator (defined in \cite{GN}). Recall that a normal subcategory $\cd$ of a fusion category of $\C$  (as defined in \cite{brn}) is the kernel of a normal tensor functor $F:\cc \ra \ce$ from  $\C$ to another fusion category $\ce$. The kernel of $F$ is the full subcategory of $\cc$ consisting of all objects whose image by $F$ is a multiple of the unit object.
\md
%\blue{As shown in \cite{brn}, given an inclusion of Hopf algebras $K \subset H$ the restriction functor $\mtr{res}^H_K$ from finite dimensional $H$-modules to finite dimensional $K$-modules is a normal tensor functor if and only if $K$ is a normal Hopf subalgebra of $H$. The restriction functor $\mtr{res}^H_K$ was  also systematically studied in \cite{Bker}  and \cite{cos}. These results generalize the results of \cite{cos} for the restriction functor of $H$-modules to $K$-modules for a normal Hopf subalgebra $K$ of a semisimple Hopf algebra $H$. }
\md
This paper is organized as follows.
%S1
In the first section we recall known facts facts about fusion categories and tensor functors which we need.
In Section \ref{coset} we define double cosets in a fusion category $\C$ relative to two fusion subcategories $\D$ and $\E$, and give the corresponding decomposition of $\C$ into indecomposable bimodule categories, and also the corresponding decomposition of the regular virtual object of $\C$ in the Grothendieck ring.
In Section \ref{eqtf}, we introduce and study two equivalence relations $\approx^F$ and $\approx_F$ attached to a tensor functor $F$. In fact  $\approx^F$ is a coset equivalence relation (Proposition \ref{lm}).
In Section \ref{ntf} we give a description of the image of simple objects under a normal tensor functor, and a new characterization of normal functors.
In Section \ref{normalcomp}, we study the situation where the composition of two tensor functors $F=F'F''$ is normal, and in particular, a criterion of normality for $F''$, with an application to equivariantizations.
In Section \ref{nfc} we introduce the  radical of a fusion subcategory, and compare it to its commutator in the case of a normal subcategory.

\section{Preliminaries}\lb{prelim}

In this section we recall the basic facts on fusion categories and tensor functors that are needed in this paper.

\subsection{Fusion categories}

Let $\kk$ be a field. A \emph{fusion category over $\kk$} is a $\kk$-linear semisimple monoidal rigid  category $\cc$ with finitely many isomorphism classes of simple objects, finite
dimensional homomorphism spaces, such that each simple object $S$ is scalar (that is, $\End(S) = \kk$) and the unit object $\un$ of $\cc$ is simple.

We refer the reader to \cite{ENO} for more details on fusion categories. A \emph{fusion subcategory} of a fusion category $\cc$ is a full replete monoidal subcategory of $\cc$
which is also a fusion category.

In a fusion category $\cc$, the left and right duals $\ldual{X}$ and $\rdual{X}$ of an object $X$ are isomorphic (but it is still not known whether there always exists a sovereign structure, that is a natural monoidal isomorphism between the two duals).

Denote by $\Lambda_{\C}$ the set of isomorphism classes of objects of  $\mtc{C}$, and by  $\Inv(\cc) \subset \Lambda_\C$ the set of isomorphism classes of invertible objects of $\C$ (which is a group for the tensor product).
 If $X$ is an object of a fusion category $\C$, we denote by $\bra{X}$ the smallest fusion subcategory of $\C$ containing $X$.

An object of $\C$ is \emph{trivial}  if it belongs to $\bra{\un}$, that is, if it is isomorphic to $\un^n$ for some $n \in \NN$.

The \emph{Grothendieck ring} $K_0(\mtc{C})$ of $\C$ is the free $\mathbb{Z}$-module
$$K_0(\mtc{C}) = \bigoplus_{ X \in \Lambda_\C } \Zed [X],$$
equipped with the product defined by the tensor product of $\C$  (if $X$ is an (isomorphism class of) simple object of $\C$, we denote by $[X]$ the corresponding basis element of $K_0(\mtc{C})$).
A \emph{virtual object of $\C$} is an element of the ring $K_0(\mtc{C})_\Comp = K_0(\mtc{C})\ot_{\mathbb{Z}}{\mathbb{C}}$, that is, a formal complex linear combination of elements of $\Lambda_\C$.

Duality induces an involution $?^*$ on $K_0(\cc)$ given by $[X]^*:=[X^*]$.

For a simple object $X$ let $\fp(X)$ denote the Frobenius-Perron dimension of $X$, that is, the Frobenius-Perron eigenvalue of the left multiplication by $[X]$ on the Grothendieck ring $K_0(\cc)$. It is a positive real algebraic number. The Frobenius-Perron dimension extends linearly to an algebra morphism $\fp: K_0(\mtc{C})_\Comp \to \Comp$.

The \emph{regular virtual object of $\C$} is the virtual object:
\begin{equation*}
    \Rg_{ _{\C}}=\sum_{X \in \SC}\fp(X)[X]
\end{equation*}
and the Frobenius-Perron dimension $\fp(\C)$ of $\cc$ is the Frobenius-Perron dimension of $\Rg_{ _{\C}}$, that is: $$\fp(\C)  :=\sum_{X \in \SC}\fp(X)^2.$$

The regular virtual object satisfies $$x \,\Rg_{ _{\C}} =\fp(x) \Rg_{ _{\C}}=\Rg_{ _{\C}} \, x$$ for any  $x \in \gr(\C)_\Comp$  (see \cite{ENO}). In particular $\rC^2=\fp(\C)\rC$.
\md
 Denote by $m_\cc$  the $\Zed$-bilinear form on the Grothendieck ring $K_0(\mtc{C})$ defined by $$m_\cc([X],[Y]) = \dim_\kk \Hom(X,Y) \quad\mbox{for $X$, $Y$ simple objects of $\C$.}$$
The bilinear form
$m_\cc$ has the following properties:
\begin{enumerate}
\item symmetry: $m_\cc(x,y) = m_\cc(y,x)$;
\item adjunction property: $m_\cc(x,y\,z) = m_\cc(y^*\,x,z) = m_\cc(x\, z^*,y)$
\end{enumerate}
for $x,y,z \in K_0(\mtc{C})$.

We also denote by $m_\cc$ the extension of this form to a $\Comp$-bilinear form on the ring $K_0(\mtc{C})_\Comp$ of virtual objects of $\C$.

Given any subset $A \subset \Lambda_\C$, we denote by $\Rg_A$ the virtual object
$$\Rg_A = \sum_{X \in A}\fp(X)[X] \in \gr(\C)_\Comp.$$
In particular $\Rg_{_\C} = \Rg_{\Lambda_\C}$.

Given a $\kk$-linear functor $G : \C \to \C'$ between fusion categories, we denote by $G_!$ the linear map $K_0(\C)_\Comp \to K_0(\C')_{\Comp}$ defined by
$$G_![X] = \sum_{Y \in \Lambda_{\C'}} m_{\C'}(Y,G(X))[Y].$$
%\beq
%m_{\C}([X], \;[Y][Z])=m_{\C}([Y]^*,\;[Z][X]^*)=m_{\C}([Z]^*,\;[X]^*[Y])\;\;\text{and}
%\eeq
%\beq
%m_{\C}([X],[Y])=m_{\C}([Y],\;[X])=m_{\C}([Y]^*,\;[X]^*)
%\eeq for all $[X],[Y],[Z] \in K_0(\mtc{C})$.
%%\subsection{Tensor functors between fusion categories}

\subsection{Tensor functors between fusion categories}\label{prelim-tensor}

\md
A \emph{tensor functor} $F : \cc \to \dd$ between two fusion categories $\cc$ and $\dd$  over a field $\kk$ is a strong monoidal
$\kk$-linear functor $F :\cc \ra \cd$.

Let $F : \cc \to \dd$ be a tensor functor between fusion categories. The \emph{kernel of $F$} is the fusion subcategory  $\KER_F \subset \C$ consisting of all objects $X$ of $\C$ such that $F(X)$ is trivial.
It is endowed with a canonical fiber functor $\omega : \KER_F \to \vect_\kk$, $X \mapsto \Hom_\D(\un,F(X))$. Hence, by Tannaka reconstruction, one obtains a finite dimensional Hopf algebra $H$ over $\kk$ such that $\KER_F \simeq \comod-H$ (see \cite{brn} for more details).

We say that $F$ is \emph{dominant} if for $Y \in \Lambda_\D$ there exists $X \in \Lambda_\C$ such that $Y$ is a factor of $F(X)$.
The \emph{dominant image of $F$} is the fusion subcategory $\D' \subset \D$ generated by the image of $F$; and $F$ is dominant if and only if $\D'=\D$. Note that $F$ can always be factorized as a dominant functor $\C \to \D'$, followed by the inclusion $\D' \hookrightarrow \D$.

We say that $F$ is \emph{normal} (a notion introduced in \cite{brn}) if for any $X \in \Lambda_\C$ such that $F(X)$ contains the unit object $\un$, $F(X)$ is trivial (that is, isomorphic to $\un^n$ for $n \in \NN$).  In that case, the Hopf algebra $H$ is called the \emph{induced Hopf algebra of $F$}.

Note that a tensor functor between fusion categories admits a left adjoint and a right adjoint. Let us denote by $R$ the right adjoint of $F$.
Then $F$ is dominant if and only if $R$ is faithful, and $F$ is normal if and only if $FR(\un)$ is trivial.
The object $A=R(\un)$ is an algebra in $\cc$, called the \emph{induced algebra of $F$.}
 Moreover, we have canonical natural isomorphisms for  $X$ in~$\C$, $Y$ in~$\D$:
$$X \otimes R(Y) \iso R(FX \otimes Y)   \quad \mbox{and} \quad R(Y) \otimes X \iso R(Y \otimes FX),$$
and in particular (taking $Y=\un)$, we have isomorphisms $A \otimes X \iso RF(X) \iso X \otimes A$.
The natural transformation $\sigma_X : A \otimes X \iso X \otimes A$ so defined is a half-braiding, and $\IA = (A,\sigma)$ is a commutative algebra in the categorical center $\zz(\cc)$ of $\cc$ called the \emph{induced central algebra of~$F$.} See \cite{blv}  for a detailed account in a more general setting.

%The following definition of a normal tensor functor was given in \cite{brn}. If $\C$ and $\D$ are fusion categories then $F:\C\ra \D$ is normal if and only if the following property is satisfied: if $m_{ \C}(1,\;F(X))>0$ then $F(X)=\fp(X)1$ for any simple object $X \in \SC$.
%\md
A fusion subcategory $\cd \subset \cc$ is \emph{normal} (in the sense of \cite{brn}) if there exists a normal tensor functor $F:\cc\ra \ce$
between fusion categories such that $\cd=\KER_F$.

\subsection{Exact sequences of fusion categories}\label{prelim-exact}

The notion of an exact sequence of tensor categories, introduced in \cite{brn} for tensor (not necessarily semisimple)  categories, generalizes the classical notion of an exact sequence of groups and the notion of an exact sequence of Hopf algebras due to Schneider. Here, we restrict our attention to fusion categories.

An \emph{exact sequence of fusion categories} is a diagram of tensor functors between fusion categories
$$(E) \quad \C' \stackrel{i}{\longrightarrow} \C \stackrel{F}{\longrightarrow} \C''$$
such that:
\begin{enumerate}
\item $i$ is fully faithful;
\item $F$ is normal and dominant;
\item the essential image of $i$ is $\KER_F$.
\end{enumerate}

We will use the following result on exact sequences (multiplicativity of Frobenius-Perron dimensions):

\begin{proposition}[\cite{brn}, Proposition 3.10]\label{multip}
Consider a diagram of tensor functors between fusion categories
$$(E) \quad \C' \stackrel{i}{\longrightarrow} \C \stackrel{F}{\longrightarrow} \C''$$
such that $i$ is fully faithful, $F$ is dominant, and $i(\C') \subset \KER_F$.
Then $$\fp(\C) \ge \fp(\C')\,\fp(\C''),$$ and equality holds if and only if $(E)$ is an exact sequence.
\end{proposition}

Interesting examples of exact sequences are provided by equivariantizations.
Let $G$ be a finite group, and $\cc$ be a fusion category. An \emph{action of $G$ on $\C$ by tensor autoequivalences} is
a strong monoidal functor  $\rho : G \to \Aut_\otimes(\C)$.

Given such an action, let $\cc^G$ be the equivariantization of $\C$ under the action of $G$.
Recall that the objects of $\C^G$ are pairs $(X,\underline{r}= (r_g)_{g \in G})$, where $X$ is an object of $\C$, and the $r_g$'s are isomorphisms $r_g:\rho(g)(X) \to X$ satisfying certain compatibilities (see  \cite{N} for more details).  The equivariantization $\C^G$ is a tensor (in general, not fusion) category, and the forgetful functor  $U_G : \C^G \to \C$, $(X, \underline{r}) \mapsto X$  is a tensor functor.

The tensor category $\cc^G$ is a fusion category provided $\kk$ is algebraically closed and its characteristic does not divide to the order of $G$.I n that case, a detailed description of the simple objects of $\C^G$ is provided in \cite{ss}.

As shown in \cite{brn}, in that situation $F$ is normal dominant, $\KER_F$ can be identified with $\rep (G)$, and we have therefore an exact sequence of fusion categories:
$$\rep(G) \longrightarrow \C^G \longrightarrow \C.$$

We say that a tensor functor $F : \C \to \D$ between fusion categories is an \emph{equivariantization} if there exists a finite group $G$ acting on
$\D$ by tensor autoequivalences, and a tensor equivalence $K : \C \to \D^G$ such that $F = U_G K$.

In \cite{brn, brn2}, several criteria are given for a tensor functor to be an equivariantization, notably in terms of central exact sequences.

An exact sequence $(E) \quad \C' \longrightarrow \C \stackrel{F}{\longrightarrow} \C''$ of fusion categories is \emph{central} if, denoting by $\IA = (A,\sigma)$ the induced central algebra of $F$, the forgetful functor $\cz(\C) \to \C$ induces an equivalence of categories $\bra{\IA} \to \bra{A}$.

In the fusion case, we have the following
\bt\label{crit-equiv}
Let $(E) \quad\C' \longrightarrow \C \stackrel{F}{\longrightarrow} \C''$ be an exact sequence of fusion categories over a field $\kk$. The following assertions are equivalent:
\begin{enumerate}[(i)]
\item The functor $F$ is an equivariantization for an action of a finite group $G$ acting on $\C''$ by tensor autoequivalences;
\item $(E)$ is a central exact sequence and its induced Hopf algebra is split semisimple.
\end{enumerate}
When these assertions hold, $H$ is commutative and $G = \Spec(H)$.
\et
\begin{proof}
This results directly from  \cite{brn2}, Proposition 3.2, combined with \cite{brn2}, Theorem 3.6.
\end{proof}

\subsection{Module and bimodule categories}

The notion of a module category is a categorification of the notion of a module over
a ring.

If $\C$ is a monoidal category, a left module category over $\C$ is a plain category $\M$ endowed with a strong monoidal functor $\rho : \C \to \underline{\End}(\M)$. In other words, it is a category $\M$ with a \emph{action} bifunctor $\ot : \C \times \M \to \M$
  with an associativity constraint $a_{U, V, M}: U \otimes (V \otimes M)\cong (U \otimes V) \ot M$ and a unit constraint $\un \otimes M \cong M$ satisfying a pentagon and a triangle coherence axiom, see  \cite{O}.

We will restrict ourselves to the case where $\C$ is a fusion category over a field $\kk$, $\M$ is a $\kk$-linear semisimple abelian category,
and the action bifunctor is $\kk$\ti linear in each variable. Such a module category is \emph{indecomposable} if $\M$ is not a direct sum of
two nontrivial module subcategories.

Let $\cc$ and $\cd$ be fusion category. Denote by $\boxtimes$ the tensor product of abelian categories introduced in \cite{d2}.
A $(\cc, \cd)$-bimodule category is a left $\cc\boxtimes \cd^{\mtr{rev}}$-module category. Here $\cd^{\mtr{rev}}$ denotes $\cd$ with opposite monoidal structure.

%__________________________________________________________
\section{Coset decompositions for fusion categories}\lb{coset}

\subsection{Double coset decomposition for fusion categories}
Let $\cc$ be a fusion category and $\D$, $\E$ be two fusion subcategories of $\C$.
Define a relation $\sim$ on $\SC$ by
$$X \sim Y \iff \exists D \in \Lam_{\D}, E \in \Lam_{\E}\mid \mbox{$Y$ is a factor of $D \ot X \ot E$}.$$

Observe that we have $X \sim Y$ if and only if $m_{\C}(X\;,\Rg_{ _{\D}}\, Y \,\Rg_{ _{\E}}) > 0$, where $\Rg_{ _{\D}}$ and $\Rg_{ _{\E}}$ denote the regular virtual objects of $\D$ and $\E$.

\begin{lemma}
The relation $\sim$ is an equivalence relation on $\SC$.
\end{lemma}

\begin{proof}
Let $X \in \SC$. Since the unit object of $\C$ belongs to $\D$ and $\E$, we have $X \sim X$, which proves the reflexivity of $\sim$. Let $X, Y \in \SC$  such that $X \sim Y$.
There exist $D \in \Lam_{\D}, E \in \Lam_{\E}$ such that $Y$ is a factor of $D \otimes X \otimes E$, that is, $m_\C(Y, D\ot Y \ot E) > 0$. By symmetry and  adjunction, $m_\C(X, D^* \ot Y \ot E^*) = m_\C(D^*\ot  Y \ot E^*,X) = m_\C(Y,D\ot X\ot E) > 0$, so $X$ is a factor of $D^* \otimes Y \otimes E^*$, and, $\D$ and $\E$ being stable by duality, that implies $Y \sim X$. So, $\sim$ is symmetric. Lastly, let $X,Y, Z \in \SC$ and assume $X \sim Y$ and $Y \sim Z$. There exist
$D, D' \in \Lam_{\D}, E, E' \in \Lam_{\E}$ such that $Y$ is a factor of $D \otimes X \otimes E$ and $Z$ is a factor of $D' \otimes Y \otimes E'$.
Therefore, $Z$ is a factor of $D' \otimes D \otimes E \otimes E'$. Now we have  $D' \otimes D = D_1 \oplus \dots \oplus D_k$ and $E \otimes E' =
E_1 \oplus \dots \oplus E_l$, where the $D_i$'s and the $E_j$'s are (isomorphy classes of) simple objects of $\D$ and $\E$ respectively, and $Z$, being simple, is a factor of one of the $D_i \otimes X \otimes E_j$'s, which shows that $X \sim Z$. Therefore $\sim$ is transitive.
\end{proof}

From now on, the equivalence relation $\sim$ will be denoted  by $r^\C_{\D,\E}$.  Its equivalence classes are called \emph{$\D$-$\E$ double cosets}.
If $B$ is a $\D-\E$ double coset, the full subcategory $\B \subset \C$ whose objects are finite direct sums of simple objects belonging to $B$ is called a \emph{$\D-\E$ double coset subcategory.}

\bn{example}  If $G$ is a finite group, denote by $\cc=\vect_{G}$  the fusion category of finite-dimensional $G$-graded vector spaces. For any two subgroups $H, L \subset G$ let $\cd:=\vect_{H}$ and $\ce:=\vect_{L}$. Then  $\cd-\ce$ double cosets in $\cc$ are just double cosets $H\backslash G \slash L$. \end{example}

\begin{rem}
The equivalence relation $r^\C_{\D,\E}$, in the special case $\cd=\vect$, appears in \cite[page 26]{dgno2}  in the study of the centralizer of a fusion subcategory of a braided fusion category.
\end{rem}

\begin{proposition}\label{bim}
Let $\C$ be a fusion category and $\D$, $\E$ be two fusion subcategories
$\C$. Let $\B_1, \dots ,\B_l$ be the $\D$-$\E$ double coset subcategories of~$\C$.
Then we have a decomposition of $\C$ as sum of indecomposable $\D$-$\E$ bimodule categories
\begin{equation*}
\C=\bigoplus_{i=1}^l\B_i
\end{equation*}
\end{proposition}

\begin{proof}
We have $\C = \bigoplus \B_i$ as $\kk$-linear categories because double cosets form a partition of $\Lambda_\C$.
The category $\C$ is a $\D$-$\E$ module category, $\D$ acting by tensoring on the left, and $\E$ on the right. The definition of $r^\C_{\D,\E}$ ensures that the $\B_i$'s are indecomposable sub-bimodule categories of $\C$.
\end{proof}

\br
It follows from \cite[ Remark 8.17]{ENO} that both $\frac{\fp(\cb_{i})}{\fp(\cd)}$ and $\frac{\fp(\cb_{i})}{\fp(\ce)}$ are algebraic integers.
\er
%
%Then $\sim$ is an equivalence relation.
%Using the above properties of the bilinear form $m$, one can see that if $X \sim Y$ then
%\begin{eqnarray*}
%  m([Y],\; R_{ _{\C}}[X]R_{ _{\D}})  &=& m(R_{ _{\C}}^*,\;[X]R_{ _{\D}}[Y]^*) =
%    m([X]^*,\;R_{ _{\D}}[Y]^*R_{ _{\C}})\\ & = &  m([X],\;R_{ _{\C}}^*[Y]R_{ _{\D}}^*)=  m([X],\;R_{ _{\C}}[Y]\rD)
%\end{eqnarray*}
%
%since $R_{ _{\C}}^*=R_{ _{\D}}$ and $R_{ _{\D}}^*=R_{ _{\D}}$. Thus $Y \sim X$.
%
%To check the transitivity suppose that $X \sim Y$ and $Y \sim Z$. Then $X$ is a subobject of $\rD(\rD Z\rE)\rE=\fp(\D)\fp(\E)\rD Z\rE$ and therefore $X \sim Z$.
\subsection{The regular virtual object of a double coset}
Let $\C$ be a fusion category, and $\D, \E \subset \C$ be two fusion subcategories.
If $B$ is a  $\D$-$E$ double coset, set
$$\gr(B)_\Comp= \oplus_{[X] \in B} \Comp [X] \subset K_0(\mtc{C})_\Comp$$ and recall the notation $$\Rg_B = \sum_{[X] \in B} \fp(X) [X] \in \gr(B)_\Comp.$$

Let $B_1, \dots, B_l$ be the list of all double cosets, so that we have
$$\gr(\C)_\Comp = \bigoplus_{i=1}^l \gr(B_i)_\Comp \quad \mbox{and} \quad \rC = \sum_{i=1}^l \Rg_{B_i},$$
and denote by $T$ the operator on  $\gr(\C)_\Comp$  defined by $T\, x =  \rD x \rE$.
\md
Recall from \cite{F} that a matrix $A \in  M_n(\mathbb{C})$ is called  \emph{indecomposable}  if the set $I = \{1,\; 2,\; \cdots ,\; n\}$  cannot  be written as a disjoint union $I = J_1 \cup J_2$  with $J_1 \neq \emptyset$ and $J_2 \neq \emptyset$, in such a way  that $a_{uv} = 0$ whenever $u \in J_1$ and $v \in J_2$.
\md
%Let $\Lam_1, \Lam_2,\cdots,\;\Lam_l$ be the the equivalence classes of the equivalence relation $r^{\C}_{ _{\D,\;\E}}$ on $\SC$
%%If $\mathcal{C}_1,\mathcal{C}_2,\cdots,\; \mathcal{C}_l$ are the equivalence classes of $r^{\C}_{ _{\D,\;\E}}$ on $\SC$ then
%and define \begin{equation}\label{pregnv}
%A_i=\sum\limits_{X \in
%{\Lam}_i}\fp(X)[X]
%\end{equation} for $1 \leq i \leq l$.
%
%For any element $X \in \Lam_{\C}$ let $L_{ _X}$ and $R_{ _X}$ be the linear operators given by the left and right
%multiplication with $[X]$ on the Grothendieck ring $\gr(\C)$.

\begin{proposition}\label{newpr}
With the above notations, for each double coset $B$ the space $\gr(B)_\Comp$ is stable under the operator $T$. Denoting by $T_B$ the restriction of $T$ to $\gr(B)_\Comp$,
the matrix of $T_B$ in the basis $B$ has real non-negative entries and is indecomposable.
Its Frobenius-Perron eigenvalue  is
$\fp(\D)\fp(\E)$ and the corresponding eigenspace is the line generated by $\Rg_B$.
\end{proposition}

% it follows that $A_i$ are eigenvectors
%of the operator $T = L_ { _{\rD}}\circ R_{ _{\rE}}$ on $\gr(C)$ corresponding to the
%eigenvalue $\fp(\D)\fp(\E)$.
%\end{proposition}

\bn{proof} The stability of $\gr(B)_\Comp$ under $T$, and the fact that the matrix of $T$ in the basis $B$ is indecomposable, are immediate consequences
of the definition of the equivalence relation $r_{ _{\D,\;\E}}^{\C}$. The entries of the matrix are real non-negative because the coefficients of
$\rD$ and $\rE$ are real non-negative. Thus, the Frobenius-Perron theorem (see \cite{F}) applies to the matrix of $T_B$: its spectral radius  $\lambda$
is an eigenvalue (the Frobenius-Perron eigenvalue), and the corresponding eigenspace  has dimension $1$ and is generated by
a vector $x$ with positive coordinates. In particular $\fp(x)$ is positive.

We have $T_B \, x = \rD \, x \, \rE = \lambda \,x$, and, taking Frobenius-Perron dimensions we obtain $\fp(\D)\fp(x)\fp(E)= \lambda \fp(x)$,
hence $\lambda = \fp(\D)\fp(E)$.

On the other hand, we have $T \rC = \rD \rC \rD = \fp(\D)\fp(\E) \rC$, and since $\rC = \sum_{i=1}^{l} \Rg_{B_i}$, we see that $\Rg_{B_i}$'s are eigenvectors
of $T$ for the eigenvalue $\lambda$. In particular $v = \Rg_B$ up to a positive scalar.
\end{proof}

\begin{corollary}
With the above notations, if $X \in B$ then
\begin{equation}\label{formula}
\rD [X]\rE=\fp(X)\,\frac{\fp(\D)\fp(\E)}{\fp(\Rg_{B})}\Rg_{B}
\end{equation}
\end{corollary}

\begin{proof} Let $y = \rD [X]\rE \in \gr(B)_\Comp$. We have $T_B y = \rD^2 [X] \rE^2 = \fp(\D)\fp(\E) y$, so by Proposition~\ref{newpr}, $y = t \Rg_B$ for some scalar $t$. We have $\fp(y)=\fp(\D)\fp(E)\fp(X) = t \fp(\Rg_B)$,
hence the value of $t$. This proves the corollary.
\end{proof}

\begin{example}\label{Hopf} Let $H$ be a finite dimensional semisimple Hopf algebra over an algebraically closed field $\kk$ of characteristic $0$. Then $H$ is also cosemisimple.
Let $K$ and $L$ be two Hopf subalgebras of  $H$.  Then the category $\C=\comod H$ of finite dimensional right $H$-comodules is a fusion category, and $\D=\comod K$ and $\E=\comod L$ are fusion subcategories of $\C$.
%Then $K$ and $L$ are also semisimple Hopf algebras; $H, K, L$ are also cosemisimple. By duality, the  inclusion $i_L:K \hookrightarrow H$ gives a Hopf algebra projection $i_L^*:H^*\ra K^*$. Thus $\mtr{Rep}(H^*)$ is a fusion category containing $\D:=\Rp(K^*)$. Similarly $\mtr{Rep}(H^*)$ contains $\E=\Rp(L^*)$.
In this case, the equivalence relation $r^{\C}_{ _{\D,\;\E}}$ coincides with the equivalence relation $r^H_{ _{K,\;L}}$ on the set of simple right comodules of $H^*$ introduced in \cite{cos}.
\end{example}
\bn{rem} The theory above allows one to extend the notion of double cosets for Hopf subalgebras from \cite{cos} to double cosets for co-quasi Hopf subalgebras of a given co-quasi Hopf algebra.
\end{rem}

We denote  by $r^{\C,\;r}_{\E}$ (resp. $r^{\C,\;l}_{\E}$) the equivalence relation $r^{\C}_{\vect,\;\E}$ (resp.  $r^{\C}_{\E,\vect}$). Its equivalence classes are called
the \emph{left (resp. right)  cosets of $\E$ in $\C$.}

\begin{example}
Let $\cc$ be a fusion category, and let
\beq\lb{grids}
\cc=\bigoplus_{g \in G}\cc_g
\eeq
be a $G$-graduation of $\cc$, for some finite group $G$. Assume that the graduation is faithful, that is, all $\C_g$'s are non-zero.
Then each $\cc_g$ is both a left and a right coset subcategory with respect to $\cc_1$.

Indeed, let $X, Y \in \Lambda_\C$. If $X$ and $Y$ belong to the same left coset,  there exists $Z \in \Lambda_{\C_1}$ such that $Y$ is a factor of $Z \otimes X$, so $X$ and $Y$ have  same degree.
Conversely, if $X$ and $Y$ have same degree, then $Y$ is a factor of $Y \otimes X^* \otimes X$, with
$Y \otimes X^*$ of degree $1$, so $X$ and $Y$ belongs to same the left coset.
This shows that the $\C_g$'s are left coset subcategories. Similarly, they are right coset
subcategories.

In particular, the fusion subcategories of $\cc$ containing $\cc_1$ are in bijection with the subgroups of $G$.
\end{example}

%\begin{example}\label{objs}
%Let $A$ and $B$ be two simple objects of $\C$. Put $\D=<A>$ and $L=<B>$ the two fusion subcategories of $\C$ generated by these objects. Then the above equivalence relation $r^{\C}_{ _{\D,\; \E}}$ can be written as follows: $X \sim Y$ if and only if
%\begin{equation}\label{onobjgen}
%    m_{\C}([X],\;[A]^n[Y][B]^m) > 0
%\end{equation}
%
%for some integers $m,n$. Recall that for a negative integer $m$ the power $[X]^{m}$ is defined as $[X^*]^m$. The above equivalence class will be denoted by $r^{\C}_{A,\;B}$. If $A=1$ denote the above equivalence relation by $r^{\C,\; r}_{B}$. Similarly define $r^{\C,\; l}_{A}$ when $B=1$.
%\end{example}
%%%%%%%%%%%%%%%%%%%%%%%%%%%%%%%%%%%%%%%%%%%%%%%%%%%%%%%%%%%%%%

%%%%%%%%%%%%%%%%%%%%%%%%%%%%%%
\section{Two equivalence relations associated with a tensor functor between fusion categories}\label{eqtf}
  Let $F : \cc \to \dd$ be a tensor functor between fusion categories, and denote by $R$ its right adjoint.

For $X \in \Lambda_\cc$ set
$$X^F = \{ [Y] \in \Lambda_\D \mid Y \,\mbox{is a factor of }\, F(X)\}$$ and
for $Y \in \Lambda_\D$, set $$Y_F = \{ [X] \in \Lambda_\C \mid Y \,\mbox{is a factor of }\, F(X)\}.$$

We have $X \in Y_F \iff Y \in X^F \iff X$ is a factor of $R(Y)$.

\bn{rem}
Note that $F$ is dominant if and only if $\Lambda_\cc^F=\Lambda_\dd$.  \end{rem}

\subsection{An equivalence relation on $\Lam_{\C}$ induced by $F$}

Define a relation $\sim^F$ on $\Lambda_\cc$ by $X  \sim^F  X' \iff X^F \cap X'^F \neq \emptyset$.
The relation $\sim^F$ is clearly reflexive and symmetric. However, it is not transitive in general:

\bn{example}\label{counterexample} Denote by $S_n$ the $n$-th symmetric group. The standard inclusion  $S_n \subset S_{n+1}$
defines by restriction a tensor functor from the category of $\Comp S_{n+1}$-modules to the category of $\Comp S_n$-modules.  It follows from Theorem 6.19 of \cite{BKK} that $\sim^F$ is not an equivalence relation.
\end{example}

Denote by $\approx^{F}$ the transitive closure of $\sim^{F}$, which is an equivalence relation on $\Lambda_\C$.

\bn{proposition}\label{lm}
Let $F:\C \ra \D$ a tensor functor between fusion categories, with right adjoint $R$. Let $\A = \bra{R(\un)}$ be the fusion subcategory of $\C$ generated by $R(\un)$.

Then $\approx^F =  r^{{C},\;r}_{\A} =r^{{C},\;l}_{\A}$. In other words, the equivalence classes of $\approx^F$ are the left cosets
with respect to $\A$, which coincide with the right cosets with respect to $\A$.
\end{proposition}

\bn{proof}
Let $A=R(\un)$. Recall from Section~\ref{prelim-tensor}  that we have for all $X$ in $\cc$: $RFX \simeq A \otimes X \simeq X \otimes A$. In particular,
the left and right cosets with respect to $\A = \bra{A}$ coincide.

On the other hand, for $X, Y \in \Lambda_\C$ we have
$$X \sim^F Y \iff \Hom_\D(FY,FX) \neq 0 \iff \Hom_C(Y,RF X) \neq 0$$ $$ \iff \Hom_\C(Y, A \otimes X) \neq 0.$$
In particular if $X \sim^F Y$, $Y$ is a factor of $A \otimes X$, so $X \,r^{\C,r}_{\A} \, Y$.  By transitivity, $X\approx^F Y \Rightarrow X \, r^{\C,r}_{\A}\, Y$.

Conversely, assume $X r^{\C,r}_{\A} Y$. That means that there exists $Z \in \Lambda_\A$ such that $Y$ is a factor of $Z \otimes X$.
Now the simple objects of $\A$ are the simple factors of $A^{\otimes n}$ and their duals. If $Z$ is a factor of $A^{\otimes n}$ with $n \geq 0$, one verifies by induction on $n$ that $X \approx^F Y$. The case where $Z$ is a dual of a factor of  $A^{\otimes n}$ reduces to the previous case, since $X$ is a factor of $Z^* \otimes Y$. Thus $ X \, r^{\C,r}_{\A}\, Y \Rightarrow X\approx^F Y$.
\end{proof}

%
% Note that
%\begin{eqnarray*}
%\mtr{Hom}_{\D}(F(X),\;F(X')) &=& \mtr{Hom}_{\D}(1,\;F(X)\ot F(X')^*)) \\ &=& \mtr{Hom}_{\D}(1,\;F(X\ot X'^*))
%\\ &=& \mtr{Hom}_{\C}(R(1),\;X\ot X'^*)\\ &=& \mtr{Hom}_{\C}(X^*,\;X'^*\ot R(1)^*)
%\\ &=& \mtr{Hom}_{\C}(X,\; R(1)\ot X')
%\end{eqnarray*}
%Thus $X\;\sim^F\;X'$ if and only if \beq\lb{q} \mtr{dim}_k\mtr{Hom}_{\C}(X,\;R(1)\ot X' )>0.\eeq Similarly it can be proven that $X\;\sim^F\;X'$ if and only if $\mtr{dim}_k\mtr{Hom}_{\C}(X,\;X'\ot R(1) )>0$ since $R(1)^*=R(1)$.
%
%
%In general $\sim^F$ is not an equivalence relation, see Remark \ref{counterexample} below for a counterexample. We let $\approx^F$ be its transitive closure. By definition this means $X\; \approx^F\; X'$ if and only if there is a sequence of simple objects $X_1,\cdots,X_n \in \SC $ such that $X=X_1 \sim^F X_2 \sim^F \cdots\sim^F X_n=X'$.
%

%%%%%%%%%%%%%%%%%%%%%%%%%%%%%%%%%%%%%%%%%%%%%%%%%%%%%%%%%%%%%%%%%%%%%%%%%%%%%%%%%%%%%%%%%%%%
\subsection{An equivalence relation on $\Lam_{\D}$ induced by $F$}\label{down}
%%%%%%%%%%%%%%%%%%%%%%%%%%%%%%%%%%%%%%%%%%%%%%%%%%%%%%%%%%%%%%%%%%%%%%%%%%%%%%%%%%%%%%%%%%%%

Similarly, define a relation $\sim_F$ on $\Lambda_\D$ by $Y \sim_F Y' \iff Y_F \cap Y'_F \neq \emptyset$. In other words, $Y \sim_F Y'$ if there exists $X \in \Lambda_\C$ such that $F(X)$ contains both $Y$ and $Y'$ as factors.
One observes immediately that $\sim_F$ is symmetric, and it is reflexive if and only if $F$ is dominant.
 In general $\sim_F$ is not transitive.

If $F$ is dominant, we denote by $\approx_F$ the transitive closure of $\sim_F$, which is an equivalence relation on $\Lambda_\D$.
%\bab Je pense que c'est faux \eab
%
%\bl\lb{d}
%Let $F :\cc \ra \cd$ a tensor functor with a right adjoint $R$. Then $Y\approx_F Y'$ if and only if there exists $n\in \NN$ such that $$\Hom_{\cd}(Y,\;FR(\un)^n\ot Y')>0.$$
%\el
%\bpf
%Note that $Y \sim_F Y'$ if and only if $\Hom_{\cc}(R(Y),\;R(Y'))>0$. On the other hand $\Hom_{\cc}(R(Y),\;R(Y'))=\Hom_{\cd}(Y,\;F(R(Y')))=\Hom_{\cd}(Y,\;R(\un)\ot Y')$. %Then the proof follows applying successively the argument.
%Then it can be iteratively observed that $Y\approx_F Y'$ if and only if $\Hom_{\cd}(Y,\;F(R(\un))^n\ot Y')>0$.
%\epf
%Let $\D_1,\cdots , \D_{l'}$ the equivalence classes of $\approx_F$ and
%\begin{equation}\label{pregnvb}
%B_i=\sum\limits_{Y \in
%\mtc{D}_i}\fp(Y)[Y]
%\end{equation}
%for $1 \leq i \leq l'$.
\begin{example}\label{hopfcase}
Consider $K$ a Hopf subalgebra of a semisimple Hopf algebra $H$ and let $F$ be the restriction functor $F:\Rp(H) \ra \Rp(K)$. Then in the paper \cite{BKK}, the equivalence relation $\approx^F$ was denoted by $u^H_{ K}$ and the equivalence relation $\approx_F$ by $d^H_K$. As explained in \cite{BKK} these equivalence relations are similar to the equivalence relations introduced by Rieffel in \cite{Ri}. They arise from the restriction functor attached to an arbitrary extension of semisimple rings in \cite{Ri}.
\end{example}
%\bn{rem}
%Suppose that $F:\C \ra \D$ is a dominant tensor functor with a right adjoint $R$. Then $R(\rD)=\rC$. This can be easily checked using the multiplicity for any simple object of $\C$ in $R(\rD)$.
%\end{rem}

\bn{proposition} \label{prop-class}
 Let $F: \C \ra \D$ be a dominant tensor functor between fusion categories. Denote by $A_1, \dots, A_l$ (resp.  $B_1, \dots, B_{l'}$) the equivalent classes of the relation $\approx^F$ (resp. $\approx_F$) on
 on $\Lambda_\C$ (resp. $\Lambda_\D)$.

 Then $l = l'$, and after reindexing the $B_j$'s we have
 $$F_! (\Rg_{A_i}) =[\C:\; \D] \Rg_{B_i} \quad \mbox{and} \quad R_!(\Rg_{B_i}) = \Rg_{A_i}.$$
for $1 \leq i \leq l$, where $[\C:\; \D]=\frac{\fp(C)}{\fp(D)}$.
\end{proposition}

\bn{proof}

The proposition results from the following two lemmas.

\begin{lemma} Let $F : \C \to \D$ be a dominant tensor functor between fusion categories. Then:
$$F_!(\Rg_{\C})=[\C:\;\D] \Rg_{\D} \quad \mbox{and} \quad R_!(\Rg_{\D})=\Rg_{\C}.$$
\end{lemma}
\begin{proof} The first identity is proved in \cite{ENO}. Now
\begin{align*}
R_!(\Rg_{\D})&=\sum_{Y \in \Lambda_\D} \fp(Y) R[Y] =
\sum_{X \in \Lambda_\C} \sum_{Y \in \Lambda_\D} \fp(Y) m_\C(X,RY) [X] \\
&=\sum_{X \in \Lambda_\C} \sum_{Y \in \Lambda_\D} \fp(Y) m_\C(FX,Y) [X]\\
&=\sum_{X \in \Lambda_\C} \fp(FX) [X] = \sum_{X \in \Lambda_\C} \fp(X) [X] = \Rg_\C.
\end{align*}
\end{proof}

\begin{lemma}\label{bijquot} Let $F : \C \to \D$ be a dominant tensor functor between fusion categories. Let
$X, X' \in \Lambda_\C$ and $Y, Y' \in\Lambda_\D$  be such that $Y$ is a factor of $F(X)$ and $Y'$ is a factor of $F(X')$. Then
$$X \approx^F X' \iff Y \approx_F Y'.$$
As a result, $F$ induces a bijection $\Lambda_\C/_{\approx^F} \to \Lambda_\D/_{\approx_F}$, defined by $A \mapsto A^F$, its inverse being defined by $B \mapsto B_F$.
\end{lemma}

\begin{proof}
The first assertion results immediately from the definition of $\approx^F$ and $\approx_F$, and the second is a direct consequence of the first and the fact that $F$ is dominant.
\end{proof}

Now let us prove the Proposition. The second lemma shows that  $l=l'$, and we renumber the $B_j$'s so that $B_i = {A_i}^F$.
Then the support of $F(\Rg_{A_i})$ is  $B_i$, and the support of $R (\Rg_{B_i})$ is   $A_i$.
The Proposition now results from the first Lemma, by restricting the linear maps induced by $F$ and $R$ to the blocks corresponding with the equivalent classes of $\approx^F$ and $\approx_F$.
\end{proof}

\bn{rem} Observe that if $F$ is not dominant, one can still apply Proposition~\ref{prop-class} to the dominant tensor functor $\C\ra \D'$, $X \mapsto F(X)$, where $\D'\subset \D$ is the dominant image of $F$.
\end{rem}
%%%%%%%%%%%%%%%%%%%%%%%%%%%%%%%%%%%%%%%%%%%%%%%%%%%%%%%%
\section{Normal tensor functors}\label{ntf}
%%%%%%%%%%%%%%%%%%%%%%%%%%%%

In this section we study the previous equivalence relations when $F$ is normal (in the sense of \cite{brn}).
Recall that a tensor functor $F : \C \to \D$ between fusion categories  is normal if for any simple object $X$ of $\C$, $m_{ \D}(\un,\;F(X))>0
 \Rightarrow F(X)$  is trivial.

\begin{remark} In other words, $F$ is normal if and only if $({(\un_{\cd})}_{F})^F=\{\un_{\cd}\}$. \end{remark}

%
%\bn{example} Consider the forgetful functor $F: \mtc{Z}(\cc) \ra \cc$ where $\mtc{Z}(\cc)$ is the Drinfeld center of the fusion category $\cc$.
%
%Then the class of unit element $1_{\cc}$under the equivalence relation $\approx_F$  is given by $\co(\cc_{ad})$, the set of simple objects of the adjoint fusion subcategory of $\cc$. Indeed if $R$ is the right adjoint of $F$ then it follows from Proposition 5.4 \cite{ENO} that $F(R(1))=\oplus_{X \in \Lam_{\cc}}X\ot X^*$.  Therefore the  the fusion subcategory generated by $F(R(1))$ is $\cc_{ad}$. Then Lemma \ref{n1} implies that $F$ is normal if and only if $\cc_{ad}=\{1_{\cc}\}$.\end{example}

Recall that $\KER_F$ is the fusion subcategory of all objects $X$ of $\C$  such that $F(X)$ is trivial (that is, $F(X)=\un^{\fp(X)}$).

\bn{theorem}\label{eq} If $F : \C \to \D$ is normal, then $\sim^F$ is an equivalence relation on $\Lambda_\C$ and coincides with $r^{{\C},\;l}_{\mtr{ker}_F}$
and $r^{{\C},\;r}_{\mtr{ker}_F}$.

If in addition $F$ is dominant, $\sim_F$ is an equivalence relation on $\Lambda_\D$.
\end{theorem}

\bn{proof}
Let $X,X' \in \Lambda_\C$. If $X \sim^F X'$, then $F(X)$ and $F(X')$ have a common simple factor $Y$. Consequently, $F(X' \otimes X^*)$ contains $Y\otimes Y^*$ which in turn contains $\un$. Now $m_\D(F(X' \otimes X^*),\un) = m_\C(X' \otimes X^*, R \un) = m_\C(X', R\un \otimes X)$ so $X'$ is a factor of $R\un \otimes X$. Since $F$ is normal, $R\un$ is an object of $\KER_F$ so $X \, r^{{\C},\;l}_{\mtr{ker}_F} \, X'$.

On the other hand, if $X \, r^{{\C},\;l}_{\mtr{ker}_F} \, X'$, there exists $E$ in $\KER_F$ such that $X'$ is a factor of $E \otimes X$. Since
$F(E) \simeq \un^n$, $F(X')$ is a factor of $F(X)^n$. This proves that ${X'}^{F}\subset X^{F}$, so $X \sim^F X'$.

This proves that $\sim^F$ is equal to $r^{{\C},\;l}_{\mtr{ker}_F}$; it is also equal to $r^{{\C},\;r}_{\mtr{ker}_F}$ by reason of symmetry, and to $\approx^F$ because it is already an equivalence relation.

Now assume $F$ is dominant. Let $Y,Y' \in \Lambda_\D$ such that $Y \approx_F Y'$. Let $X, X' \in \Lambda_\C$ such that $Y \in X^F$ and $Y' \in X'^F$. Then we have $X \approx^F X'$ by Lemma~\ref{bijquot}, so $X \, r^{{\C},\;l}_{\mtr{ker}_F} \, X'$. We have just seen that this implies $X'^F \subset X^F$, so
$Y, Y' \in X^F$, hence $Y \sim_F Y'$. This proves that $\sim_F\, = \,\approx_F$ is an equivalence relation.
\end{proof}

\begin{corollary} Let $F:\cc \ra \dd$ be a tensor functor between fusion categories and let $\dd'$ be its dominant image. The following assertions are equivalent:
\begin{enumerate}[(i)]
\item $F$ is normal;
\item for $X,X' \in \Lambda_\cc$, $X^F$ and $X'^F$ are either disjoint
or equal;
\item for $Y,Y' \in \Lambda_{\dd'}$, $Y_F$ and $Y'_F$ are either disjoint
or equal;
%\item for $X,X' \in \Lambda_\cc$, $(X^F)_F$ and $(X'^F)_F$ are either
\end{enumerate}
\end{corollary}

\begin{proof} We may assume $F$ dominant and therefore $\D = \D'$.

(ii)  $\implies$ (i) : let $X \in \Lambda_\C$ such that $\un_{\D} \in X^F$. Since $\un_{\C}^{\;F}=\{\un_{\D}\}$, (ii) implies that
$X^F=\{\un_{\D}\}$, that is $F(X)$ is trivial.
Thus, $F$ is normal.

(iii) $\implies$ (i) : let $X \in \Lambda_\C$ such that $\un_\D \in X^F$, and let $Y \in X^F$.  We have $X \in (\un_\D)_F \cap Y_F$, so (iii) implies
$Y_F = (\un_\D)_F$. In particular $\un_\C \in Y_F$, which means $Y= \un_\D$. This shows that $F(X)$ is trivial. Thus, $F$ is normal.

(i) $\implies$ (ii) and (iii) because if $F$ is normal, by Theorem~\ref{eq} the $X^F$'s and the $Y_F$'s are equivalence classes for $\sim^F$ (resp. $\sim_F)$.
\end{proof}

The above Corollary can be regarded as an analogue of the fact that a Hopf subalgebra is depth two if and only if it is normal. See \cite{BKK} for a proof in the context of Hopf algebras. %More generally the depth of a functor was introduced in \cite{CK}

%%%%%%%%%%%%%%%%%%%%%%%%%%%%%%%%%%%%%%%%%%%%%%%%%%%%%%%%%%%%%%%%%%%%%%%%%%%%%%%%%%%%%%%%%%%%
%\subsection{Description of the image of a normal tensor functor}
%%%%%%%%%%%%%%%%%%%%%%%%%%%%%%%%%%%%%%%%%%%%%%%%%%%%%%%%%%%%%%%%%%%%%%%%%%%%%%%%%%%%%%%%%%%%
%
%Let as before $\mathcal{C}_1,\mathcal{C}_2,\cdots,\; \mathcal{C}_l$ be
%the equivalence classes of $r^{\C,\;l}_{ _{\mtr{Ker}_F}}$ on $\SC$ and let \begin{equation}\label{pregnv}
%A_i=\sum\limits_{X \in
%\mtc{C}_i}\fp(X)[X]
%\end{equation} for $1 \leq i \leq l$.

\bt\label{image}
Let $F:\C \ra \D$ be a normal tensor functor between fusion categories, and denote by  $R$ its right adjoint. Let
$A_1, \dots, A_l$ denote the equivalence classes of $\approx^F$ in $\Lambda_\C$,

\begin{enumerate}

\item For $X \in A_i$, $$\frac{F_![X]}{\fp(X)} = \frac{ F_!(\Rg_{A_i})}{\fp(\Rg_{A_i})},$$

\item for  $X,\;X' \in \SC$:
\begin{equation*}
    X \approx^F X' \iff \frac{F_![X]}{\fp(X)}= \frac{F_![X']}{\fp(X')}.
\end{equation*}

\end{enumerate}
Now assume in addition that $F$ is dominant, and let $B_1, \dots, B_l$  be the
equivalence classes of $\approx_F$ in $\Lambda_\D$, with $B_i = (A_i)_F$. Then:
\begin{enumerate} \setcounter{enumi}{2}
\item
for $Y \in B_i$,
$$\frac{R_![Y]}{\fp(Y)} = [\C:\D] \frac{\Rg_{A_i}}{\fp(\Rg_{A_i})},$$
\item for $Y,\;Y' \in \SD$:
\begin{equation*}
     Y \approx_F \;Y' \iff \frac{R_![Y]}{\fp(Y)}= \frac{R_![Y']}{\fp(Y')}.
\end{equation*}
\end{enumerate}

\et

\bn{proof}
1) Let $X \in A_i$. By Theorem~\ref{eq}, the $A_i$'s are the left cosets relative to the fusion subcategory $\KER_F$. In particular, the class of $\un$ is
$\KER_F$.
Relation \eqref{formula} implies that
$$ \Rg_{\KER_F} [X] = \fp(X){\fp({\KER_F})\over\fp(\Rg_{A_i})} \Rg_{\A_i}.$$
Applying $F_!$, and noting that $F_!(\Rg_{\KER_F}) = \fp(\KER_F) [\un]$, we obtain the formula of Assertion (1).

2) An immediate consequence of  the first assertion is that  $$X \approx^F X'\Rightarrow\fp(X)^{-1} F_![X] = \fp(X')^{-1} F_! [X'].$$
On the other hand,
if $X, X' \in \Lambda_\C$ are such that $F_![X]$ and $F_![X']$ are colinear in $\gr(\D)_\Comp$, then $F(X)$ and $F(X')$ have a common simple factor so $X \sim^F X'$. This proves Assertion (2).

3)  Let $Y \in B_i$. We have $Y^F = A_i$, and
\begin{eqnarray*}
R_![Y] &=& \sum_{X \in \SC}m_{\C}(X,R(Y)) [X] = \sum_{X \in \SC}m_{\D}(F(X),Y)[X]\\&=&
\sum_{X \in A_i}m_{\D}(F(X),Y)[X]  \\ &=&
  \sum_{X \in A_i}m_{\D}(\frac{F(X)}{\fp(X)},Y)\fp(X)X  \\ &=& m_{\D}(Y,\; \frac{F(\Rg_{A_i})}{\fp(A_i)})\sum_{X \in A_i}\fp(X)X
  \\&=&  m_{\D}\bigl(Y,\; \frac{[\C:\D]\,\Rg_{B_i}}{\fp(A_i)}\bigr)\sum_{X \in A_i}\fp(X)X
  \\&=&[\C:\;\D]\, \fp(Y)\,\frac{\Rg_{A_i}}{\fp(A_i)},
\end{eqnarray*}
hence Assertion (3).

4) Assertion (3) shows that  $$Y \approx_F Y'\Rightarrow\fp(Y)^{-1} R_![Y] = \fp(Y')^{-1}R_![Y'].$$
On the other hand if  $R_![Y]$ and $R_![Y']$ are colinear virtual objects, then $R(Y)$ and $R(Y')$  have a common simple factor so $Y \sim_F Y'$. This proves Assertion (4).
\end{proof}
%As a consequence of that, if $F$ is normal then $F$ induces a bijection
%$\Lambda_\cc/\simeq^F \iso \Lambda_{\dd'}/\simeq_F$.
%\subsubsection{Cosets for normal fusion subcategories}
Recall \cite{brn} that a fusion subcategory $\cd \subset \cc$ is called normal if there is a normal tensor functor $F:\cc\ra \ce$ such that $\cd=\KER_{F}$.

The following Proposition, which is a straightforward consequence of Theorem~\ref{eq} and \ref{formula}, can be seen as a generalization of the fact that left and right cosets of a normal subgroup coincide. Its analogue for Hopf algebras was proven in \cite{cos}.

\begin{corollary}\label{leqr}
If $\D$ is a normal fusion subcategory of $\C$ then the left and right cosets of $\cc$ relative to $\cd$ coincide.  Moreover,
$\Rg_{\D}$ is then central in $\gr(\C)_\Comp$.
\end{corollary}
\section{Normality and composition of functors, with an application to equivariantizations}\label{normalcomp}

In this section, we study the following question: if the composition of two tensor functors is normal, what can be said about each functor?
We apply our result to the special case of  equivariantizations.

\begin{theorem}\label{thm-normal}
Consider a commutative triangle of tensor functors between fusion categories:
$$\xymatrix{
\E \ar[rr]^{F''}\ar[rd]_F && \D \ar[ld]^{F'}\\
& \C &
}$$
and assume that $F$ is normal and $F''$ is dominant.
Then
 \begin{enumerate}
\item
The functor $F'$ is normal;

\item Denoting by $\IA$ and $\IA''$ the central induced algebra of $F$ and $F''$, which are commutative algebras in the categorical center $\zz(\C)$ of $\C$, $\IA''$ is a subalgebra of $\IA$.

\item The functor $F''$ induces a dominant tensor functor $$F''_0: \KER_F \to \KER_{F'},$$ with $\KER_{F''_0} =\KER_{F''}$, and the following assertions are equivalent:
\begin{enumerate}[(i)]
\item $F''$ is normal;
\item $F''_0$ is normal;
\item $\fp(\KER_{F''}) = [\KER_F: \KER_{F'}]$.
\end{enumerate}
\end{enumerate}

\end{theorem}
\begin{remark}
The first assertion of Theorem~\ref{thm-normal} generalizes the classical fact that if a subgroup is normal, then it is normal in any intermediate subgroup. \end{remark}

\begin{remark} Denote by $H$ and $K$ the induced Hopf algebras of the normal functors $F$ and $F'$ respectively, so that we may identify
$\KER_F$ to $\comod(H)$ and $\KER_{F'}$ to $\comod(K)$; then the dominant tensor functor $F''_0$ is induced by a surjective Hopf algebra morphism $p : H \to K$, and condition (ii)  of Assertion (3)  of the Theorem means that $K^*$ is a normal Hopf subalgebra of $H^*$ (see~\cite{brn}).
\end{remark}

\begin{proof}  Assertion (1).  Let $Y \in \Lambda_\D$ and assume that $F'(Y)$ contains $\un$. Since $F''$ is dominant, there exists $X \in \Lambda_\E$ such that $F''(X)$ contains $Y$. Therefore, $F'F''(X) = F(X)$ contains $\un$, and since $F$ is normal, $F(X)$ is trivial. So $F'(Y)$, being contained in  $F'F''(X)$,
is trivial too. This shows that $F'$ is normal.

 Assertion (3). Clearly,  $F''(\KER_F) \subset \KER_{F'}$;  so $F''$ induces by restriction a tensor functor $F''_0 : \KER_F \to \KER_{F'}$. The kernel of $F''_0$ is $\KER_{F''}$.

\emph{Key fact:} if $X \in \Lambda_\E$ is such that $F''(X)$ contains $Y \in \Lambda_{\KER_{F'}}$, then $X$ belongs to $\KER_{F}$. Indeed, in that case
$F(X) = F'F''(X)$ contains $F'(Y) \simeq \un^{\fp(Y)}$, so, $F$ being normal, $F(X)$ is trivial.

From the key fact, we draw two consequences:

1) If $X \in \Lambda_\E$ is such that $F''(X)$ contains $\un$, then $X \in \KER_F$.
This shows that $F''$ is normal if and only if $F''_0$ is normal, thus (i) $\iff$ (ii).

2) The functor $F''_0$ is dominant. Indeed, for $Y \in \Lambda_{\KER_{F'}}$, there exists $X \in \Lambda_\C$ such that $F(X)$ contains $Y$, because $F''$ is dominant; and the key fact insures that $X \in \KER_F$.

Denote by $i$ the inclusion $\KER_{F''} \hookrightarrow \KER_F$, and consider the sequence of tensor functors:
$$(E)\quad \KER_{F''} \stackrel{i}{\longrightarrow} \KER_F \stackrel{F''_0}{\longrightarrow} \KER_{F'},$$
where $F''_0$ is dominant and $\KER_{F''}$ is its kernel.

We may apply Proposition~\ref{multip}: we have $$\fp(\KER_F)  \ge  \fp(\KER_{F'}) \fp(\KER_{F''}),$$ and equality holds if and only if $(E)$ is an exact sequence, that is, $F''_0$ is normal. This shows (ii) $\iff$ (iii).

Assertion (2). Recall that the induced central algebra $\IA$ of $F$ is defined as follows. See \cite{brn} for details, and \cite{blv} for a more general account (with the dual point of view). Denote by $R'$, $R''$ the right adjoints of $F'$ and $F''$ respectively, so that $R=R''R'$ is right adjoint to $F$. Let $\hat{T}=RF=R''R'F'F''$ be the monad of the Hopf monoidal adjunction $(F,R)$. Then $\hat{T}$ is a monoidal monad on $\E$.
Let $A = \hat{T}(\un)$.
For $X$ an object of $\E$, define morphisms $u_X : A \otimes X \to \hat{T}(X)$ and $v_X : X \otimes A \to \hat{T}(X)$ by the commutativity of the following diagrams:
$$
\xymatrix @!0 @R=1.2cm @C=3.5cm{
    A \otimes X \ar[r]_{A \otimes \eta_X}\ar@/^1pc/[rr]^{u_X}& A \otimes \hat{T}(X) \ar[r]_{\hat{T}_2(\un,X)} & \hat{T}(X)\\
    X \otimes A \ar[r]_{\eta_X \otimes A}\ar@/^1pc/[rr]^{v_X}& \hat{T}(X) \otimes A\ar[r]_{\hat{T}_2(X,\un)} & \hat{T}(X),
}$$
where $\eta$ is the unit of the monad $\hat{T}$, and
$\hat{T}_2$ is its monoidal structure.  Then $u_X$ and $v_X$ are in fact isomorphisms (because the adjunction $(F,R)$ is Hopf). Define an isomorphism $\sigma : A \otimes \id_\E \to \id_\E \otimes A$  by $$\sigma_X = v^{-1}_X u_X.$$
Then $\sigma$ is a half-braiding in $\E$, and $\IA = (A,\sigma)$.

Similarly, $\IA''=(A'',\sigma'')$ is defined in terms of the monoidal monad $\hat{T''} = R'' F''$ of the Hopf monoidal adjunction $(F'',R'')$.

Now the unit $\eta' :1_{\cd}\ra  R'F' $ of the monoidal adjunction $(F',R')$
defines a natural transformation $f = R''\eta'T'': \hat{T''} \to \hat{T}$ which is clearly a monoidal morphism of monads.
In particular, $f_\un$ is an algebra morphism $A'' \to A$. Moreover, since $\sigma$ and $\sigma''$ are defined in terms of the monoidal monad structures
of $\hat{T}$ and $\hat{T''}$, which are preserved by $f$, one verifies easily that
 $(X  \otimes f_\un) \sigma''_X = \sigma_X (f_\un \otimes X)$, that is, $f_\un$ is an algebra morphism from $\IA''$ to $\IA$ in $\cz(\cc)$.

%so $f_\un$ is an algebra morphism from $\IA''$ to $\IA$.

Lastly, $f$ is a monomorphism because $F'$ being faithful, $\eta' : \id_\D \to R'F'$ is a monomorphism and $R''$ preserves monomorphisms. In particular,
$\IA''$ is a subalgebra of $\IA$ via $f_\un$.
%Her
%
%$$\xymatrix{
%\\
%B \otimes X \ar[r]_{B \otimes \eta''_X}\ar[d]_{f_\un \otimes X} \ar@/^3pc/[rrrr]^{\sigma''_X}\ar@/^1pc/[rr]^{\simeq}
%&B\otimes \hat{T''}(X)\ar[r]_{\quad\hat{T''}_2(\un,X)}\ar[d]_{f_\un \otimes f_X} & \hat{T''}(X)\ar[d]_{f_X} &\hat{T''}(X) \otimes B \ar[l]^{\hat{T''}_2(X,\un)\quad}\ar[d]^{f_X \otimes f_\un} & X \otimes B \ar[l]^{\quad\eta''_X \otimes B}\ar[d]^{X \otimes f_\un} \ar@/^-1pc/[ll]_{\simeq}\\
%A \otimes X \ar[r]^{A \otimes \eta_X} \ar@/^-3pc/[rrrr]_{\sigma_X} \ar@/^-1pc/[rr]_{\simeq}&A\otimes \hat{T}(X)\ar[r]^{\quad\hat{T}_2(\un,X)} & \hat{T}(X)  &\hat{T''}(X) \otimes A \ar[l]_{\hat{T}_2(X,\un)\quad} & X \otimes A \ar[l]_{\quad\eta_X \otimes A}\ar@/^1pc/[ll]^{\simeq}\\
%\\},$$
%where $\eta$, $\eta''$ denote the units, and $\hat{T}_2$, $\hat{T''}_2$ the monoidal structures, of the monoidal monads $\hat{T}$ and $\hat{T''}$ respectively. The commutativity of the upper part (resp. of the lower part) of this diagram is the very definition of the half-braiding $\sigma''$ (resp. $\sigma$). As for the four squares in the middle, the outer ones are commutative because $f$ is a monad morphism, and the inner ones, because $f$ is monoidal. As a result, we obtain: $(X  \otimes f_\un) \sigma''_X = \sigma_X (f_\un \otimes X)$, that is, $f_\un$ is an algebra morphism from $\IB$ to $\IA$.
%\bab\eab
\end{proof}

We can apply Theorem~\ref{thm-normal} to equivariantizations  (see Section~\ref{prelim-exact}).

\begin{corollary} Let $G$ be a finite group acting on a fusion category $\C$ by autoequivalences. Assume that the base field $\kk$ is algebraically closed, and its characteristic does not divide the order $|G|$ of $G$, so that the equivariantization $\C^G$ is a fusion category.
Let $H$ be a subgroup of $G$. Then the restriction functor $r^G_H:\C^G \to \C^H$ is normal if and only if $H$ is a normal subgroup of $G$.

Moreover, in that case the group $G/H$ acts on $\C^H$ by tensor autoequivalences, and $\C^G$ is tensor equivalent to $(\C^H)^{G/H}$ in such a way that the following diagram of tensor functors commutes:
$$\xymatrix{
\C^G \ar[rr]^{\simeq_\otimes}\ar[rd]_{r^G_H} && {(\C^H)}^{G/H} \ar[ld]^{U_{G/H}}\\
& \C^H &
}$$
\end{corollary}

\begin{proof}
We apply Theorem~\ref{thm-normal} to $\C = \C$, $\D= \C^H$, $\E = \C^G$, $F :\C^G \to \C$, $F': \C^H \to \C$ are the forgetful functors, and $F''=r^G_H:\C^G \to \C^H$ is the restriction functor defined by sending an object $(X,(r_g)_{g \in G})$ of $\C^G$ to the object $(X,(r_h)_{h \in H})$ of~$\C^H$.

In this case, $\KER_F= \rep(G)$, $\KER_{F'}= \rep(H)$, and the functor $$F''_0 : \KER_F \to \KER_{F''}$$ is just the restriction functor $\rep(G) \to \rep(H)$. Thus $\KER_{F''_0}=\KER_{F''} = \rep(G/\overline{H})$, where $\overline{H}$ is the normal subgroup of $G$ generated by $H$.
The equivalence of (i) and (iii)  in Assertion (3) tells us that $F''$ is normal if and only if $\fp(\rep(G/\overline{H}))=\frac{\fp(\rep(G))}{\fp(\rep(H))}$, that is,
$|G/\overline{H}| = |G/H|$, which is equivalent to $\overline{H} = H$, in other words $H$ is normal in $G$.

Now assume $H$ is normal in $G$. In addition to the exact sequence:
$$(E)\quad \KER_{F}=\rep(G) \longrightarrow \C^G \stackrel{F}{\longrightarrow} \C,$$
we have a second exact sequence of fusion categories:
$$(E'')\quad \KER_{F''}=\rep(G/H) \longrightarrow \C^G \stackrel{F''}{\longrightarrow} \C^H.$$

The induced Hopf algebra of $(E)$ is $\kk^G$, and that of $(E'')$ is $\kk^{G/H}$.
According to the equivariantization criterium of Theorem~\ref{crit-equiv}, $(E)$ is a central exact sequence.
Let $\IA = (A,\sigma)$ and $\IA''=(A'', \sigma'')$ denote the central induced algebras of $F$ and $F''$ respectively.
Centrality of $(E)$ means that the forgetful functor $\bra{\IA} \to  \bra{A}$ is an equivalence of categories.
By Assertion (2) of Theorem~\ref{thm-normal}, $\IA'$ is a subalgebra of $\IA$, so that $\bra{\IA''}$ is a fusion subcategory of $\bra{\IA}$. The forgetful functor $\bra{\IA''} \to \bra{A''}$ is full and dominant, so it is an equivalence, which means that $(E'')$ is central.
Now, again by Theorem~\ref{crit-equiv}, $(E'')$ is an equivariantization exact sequence, with group $G/H$, and we are done.
\end{proof}
%\bll{\br Note that $T^{\bar g}(M)\cong T^{g}(M)$ for any $M \in \cc^{H}$. \er}

\section{On the radical and commutator of a normal fusion subcategory}\lb{nfc}
In this section we introduce the radical of a fusion subcategory and compare it to the commutator in the case of a normal fusion subcategory.

%\bl\lb{int}
%If $\cd$ is a normal fusion subcategory of $\cc$ then $\cd\cap \cc_1$ is a normal fusion subcategory of $\cc_1$ for any fusion subcategory $\cc_1$ of $\cc$.
%\el

%\subsection{On the commutator and the radical of a normal fusion subcategory }

Let $\C$ be a fusion category, let and $\mtc{D}$ be a fusion subcategory of $\mtc{C}$.

The \emph{radical of $\D$ in $\C$,} denoted by $\mtr{rad}_{\cc}(\cd)$, is the full abelian subcategory of $\C$ generated by the simple objects $X$ of $\C$  such that $X^{\otimes n}$ belongs to $\D$ for some integer $n>0$.

The \emph{commutator of $\D$ in $\C$}, denoted by $\cd^{co}$, is the full abelian subcategory of $\C$ generated by the simple objects $X$ of $\C$ such that $X\otimes X^*$ belongs to $\D$. This notion is introduced in \cite{GN}.
%If $K_0(\cc)$ is commutative then  \blue{the commutator} is a fusion subcategory of $\cc$  (but not in general).

Note that if $K_0(\cc)$ is commutative (for example $\cc$ braided) then $\cd^{co}$ and  $\rad_{\cc}(\cd)$ are fusion subcategories of $\C$, but not in general.

For a normal fusion subcategory the radical and the commutator coincide. Indeed:

\bp\lb{comm}
Let $\C$ be a fusion category, and let $\D \subset \C$ be  a normal fusion subcategory of $\cc$. Denote by $F : \C \to \tilde{\C}$ a tensor functor between fusion categories such that $\D = \KER_F$.

Let $\E$ be the full abelian subcategory of $\C$ generated by the simple objects  $X$ of $\C$ such that $F(X) = M^n$, with $M$ invertible and $n \ge 0$. Then $\cd^{co} = \rad_{\cc}({\cd}) = \E$.
In particular, $\cd^{co}$ and $\rad_{\C}({\D})$ are fusion subcategories of $\C$.
\ep
\bpf
Observe that $\E$ is a fusion subcategory, because clearly a simple factor of the tensor product of two simple objects of $\E$, $\un$, and the dual of a simple object of $\E$, all belong to $\E$.

In order to show the inclusion $\D^{co} \subset \E$, we will need the following

\begin{lemma} Let $\C$ be a fusion category, and $X \in \Lambda_\C$. The following assertions are equivalent:

\begin{enumerate}[(i)]
\item $X$ is a multiple of an invertible object of $\C$;
\item There exists $n > 0$ such that $X^{\otimes n}$ is trivial.
\end{enumerate}
\end{lemma}
\begin{proof}
(i) $\implies$ (ii) because $\Inv(\C)$ is a finite group, so invertible objects of $\C$ have finite order.
(ii) $\implies$ (i). Assume $X^{\otimes n}$ is trivial for some $n > 0$, and let us show that $X$ is a multiple of an invertible object. We begin with the case where $X$ is simple. Since
$\un$ is a factor of $X \otimes X^{\otimes n-1}$, $X^*$ is a factor of  $X^{\otimes n-1}$, $X \otimes X^*$ is a factor of $X^{\otimes n}$ and in particular $X \otimes X^*$ is trivial.
Now $m_\C(\un, X \otimes X^*) = m_\C(X,X) = 1$,
so $X \otimes X^* = \un$, that is, $X$ is invertible. Now for the general case. Let $M$ be a simple factor of $X$. Since $X^{\otimes n}$ contains $M^{\otimes n}$, $M$ is invertible. Let $N$ be another simple (also, invertible) factor of $X$. Then $M^{\otimes n}$ contains $M^{\otimes n}$ and $N \otimes M^{\otimes n-1}$, which are both trivial and simple and therefore isomorphic to $\un$. This implies $M \simeq N$, so $X$ is a multiple of $M$.
\end{proof}

Now let $X$ be a simple object of $ \rad_{\C}({\D})$. There exists $n > 0$ such that $F(X)^{\otimes n}$ is trivial, and by the Lemma, $F(X)$ is a multiple of
an invertible object of $\tilde{\C}$, so $X$ belongs to $\E$. Thus, $\rad_{\C}({\D}) \subset \E$.

Now if $X$ be a simple object of  $\D^{co}$, then $F(X) \otimes F(X)^*$ is trivial. If $M$, $N$ are two simple factors of $F(X)$, $M \otimes N^*$ is trivial, which implies that $M = N$ is invertible, and $X$ lies in $\Lambda_\E$. Thus $\D^{co} \subset \E$.

Conversely, let $X$ be a simple object of $\E$. We have $F(X) = M^m$, with $M$ invertible. By the lemma, there exists $n > 0$ such that $F(X)^{\otimes m}$ is trivial, so $X^{\otimes m}$ belongs to $\D$, and $X$ belongs to $\rad_{\C}(\D)$.
Also, $F(X\otimes X^*)$ is trivial, so $X \otimes X^*$ belongs to $\D$, and $X$ belongs to $\D^{co}$.

Thus, $\rad_{\C}({\D})=\D^{co} = \E$.
\epf
%
%\bl\lb{ress}
%Let $F:\cc\ra \cd$ is be tensor functor between fusion categories, and let $X, Y \in \Lambda_\cc$. Then  $X\ot Y \in \KER_F$ if and only if there exists an invertible object $M \in \Inv(\cc)$ such that\beq
%\frac{F(X )}{ \fp(X)}=\frac{F(Y^* )}{ \fp(Y^*)}=M.
%\eeq
%\el
%\bpf The `if' part is clear. Now assume $X \otimes Y \in \KER_F$, that is, $F(X) \otimes F(Y)^*$ is trivial. Let $M$ be any simple factor of $F(X)$, and $N$ be any simple factor of $F(Y)$. Then $M \otimes N^* = 1^n$. Now $n = m_\C(1,M\otimes N*)=m_\C(N,M) = 1$ if $N = M^*$, $0$ otherwise. So $N = N^*$, $n=1$, and $M$ is invertible. Also, all simple factors of $F(X)$ are isomorphic to $M$, and all simple factors of $F(Y)$, to $M^*$. Hence
%$F(X) = \fp(X) M$ and $F(Y) = \fp(Y) M^*$, which proves the `only if' part.
%\epf

%\blue{ Lemma with $res{\al\beta}$ for normal fusion categories}

%\br\er\lb{costs}
%\br Suppose that $\cc$ is a normal $G$-extension of a fusion subcategory $\cd$. The proof of the previous Proposition also shows that $\cd^{co}=\cc$.  Examples of  normal $G$-extension are given by $\rep(A//K)\subset \rep(A)$, where $K$ is a central Hopf subalgebra of $A$.
%\er
\subsection*{Acknowledgments} The work of Sebastian Burciu  was supported by a grant of the Romanian National Authority for Scientific Research, CNCS - UEFISCDI, project number PN-II-RU-TE-2012-3-0168.
\bibliographystyle{amsplain}
\bibliography{ntf1}
\ed